\documentclass[12pt,a4paper]{article}

\usepackage[whole]{bxcjkjatype}

\usepackage{amsmath}
\usepackage{amssymb}
\usepackage{bussproofs}
\usepackage{tikz}
\usepackage{stmaryrd}
\usepackage{cmll}
\usepackage{mathptmx}
\usepackage{xcolor}
\usepackage{amsthm}
\usepackage{framed}
\usepackage{here}
\usepackage{comment}

\newcommand{\tensor}{\mbox{$\otimes$}}
\newcommand{\pa}{\bindnasrepma}

\newtheorem{defn}{Definition}
\newtheorem{lem}{Lemma}
\newtheorem{thm}{Theorem}

\newtheorem{cor}{Corollary}

\newtheorem{rem}{Remark}

\fontsize{14pt}{0pt}\selectfont
\title{Deducibility of Identicals, Reflection Principle and Synthetic Connectives}
\fontsize{12pt}{0pt}\selectfont
\author{Yuki Nishimuta}
\fontsize{12pt}{0pt}\selectfont
\date{}
					
\begin{document}
\maketitle

\begin{abstract}
Sambin et al. (2000) introduced Basic Logic as a uniform framework for various logics.  At the same time, they also introduced the principle of reflection as a criterion for being a connective in Basic Logic. In this paper, we make explicit the relationship between Hacking's deducibility of identicals condition (Hacking, 1979) and the principle of reflection by proving their equivalence.  Moreover, despite Sambin et al.'s conjecture that only six connectives satisfy the principle of reflection, we show that a logical connective satisfies principle of reflection if and only if it is Girard's synthetic connective.


\end{abstract}

\section{Introduction}   

In this paper, we investigate the deducibility of identicals condition. The deducibility of identicals condition (abbreviated as DoI) was introduced by Hacking (1979) as a deducibility condition along with transitivity (i.e. the admissibility of the cut rule) and dilution (i.e. the admissibility of the weakening rule).  The importance of transitivity condition is clear from the well-known example of Prior's tonk connective (Belnap, 1962). On the contrary, such a persuasive example is not known in the case of the DoI condition and its  necessity  has not yet been fully clarified.

While the transitivity condition has been extensively discussed in the literature, relatively few studies have focused on the DoI condition. The aim of this paper is to provide a technical basis for further discussion of the DoI condition by elucidating its relationship with another criterion for logical connectives, namely Sambin et al.'s principle of reflection (reflection principle).

Sambin et al. introduced Basic Logic as a uniform framework encompassing various logics.  At the same time, they introduced the principle of reflection, as a requirement that logical connectives of Basic Logic must satisfy. In their approach, logical connectives are obtained by solving definitional equations that relate logical connectives to meta-linguistic links between assertions.

It has been suggested that a connective $\mathcal{C}$ satisfies Sambin et al.'s principle of reflection if and only if $\mathcal{C}$ satisfies DoI and the main step of the cut elimination (Schroeder-Heister, 2013).
However, this equivalence has not been fully articulated in the literature. In this paper, we establish this equivalence within our formal setting.

In order to clarify the relationship between the DoI condition and the principle of reflection, we reinterpret the reflection principle from a different perspective.  Sambin et al. begin with a definitional equation and then construct inference rules from it (Sambin et al., 2000, p.983). By contrast, we start with arbitrary inference rules, construct a corresponding definitional equation, and then attempt to solve it following their procedure. Sambin et al. do not discuss cases in which a definitional equation is unsolvable, that is, cases in which their procedure fails. We make explicit the conditions under which such unsolvability arises.

Finally, in order to characterize which logical connectives satisfy the principle of reflection and which do not, we consider $n$-ary connectives rather than restricting attention to binary connectives.

The main contribution of this paper is to make explicit the precise relationship between Hacking's deducibility of identicals condition and Sambin et al.'s principle of reflection. We show that these two criteria are equivalent once the main step of cut elimination is assumed, thereby clarifying a connection that has previously only been suggested in the literature. Furthermore, by reversing Sambin et al.'s direction - from inference rules to definitional equations rather than vice versa - we reveal that solvability of definitional equations is not an intrinsic property of a connective, but can depend on its presentation. This observation leads to an exact characterization of when definitional equations are solvable: namely, when the connective has a singleton, non-context-changing formation rule satisfying visibility. As a consequence, we prove that a logical connective satisfies the reflection principle if and only if it is a Girard's synthetic connective.

\section{Preliminaries}

\subsection{Deducibility of Identicals}
We briefly explain the deducibility of identicals, the principle of reflection and synthetic connectives in this chapter. The deducibility of identicals condition is formalized in (Naibo and Petrolo, 2015, p.147) as follows.

We must be able to show that logical constants are uniquely identified by their inference rules: given an $n$-ary operator $\mathcal{C}$ and arbitrary formulas $A_1,\dots,A_n$, the sequent  $\mathcal{C}(A_1,\dots,A_n)\vdash\mathcal{C}(A_1,\dots,A_n)$ has to be derivable using at least one rule for $\mathcal{C}$ and, when needed, only the identity axiom to close the derivation.

Whether this condition is satisfied is checked by the following procedure:

Start by applying the left (resp. right) rule(s) of the operator under analysis, and immediately conclude by applying its right (resp. left) rule(s)  (Naibo and Petrolo, 2015, p.147).

Thus, a procedure for deciding the satisfiability of the deducibility of identicals condition amounts to a proof search. The DoI condition is strictly weaker than the uniqueness condition (Naibo and Petrolo, 2015, pp.153-154). The uniqueness condition states that for two logical connectives $\mathcal{C}, \mathcal{C}'$ with the same inference rules, both $\mathcal{C}\vdash\mathcal{C}'$ and $\mathcal{C}'\vdash\mathcal{C}$ are deducible (Belnap, 1962).

The DoI condition corresponds to the $\eta$-expansion in the presence of structural rules. However, this correspondence fails in the absence of structural rules: for example, the tensor connective satisfies the DoI condition but does not satisfy $\eta$-expansion.

\subsection{Principle of Reflection}

The principle of reflection states that a logical connective reflects a meta-linguistic link between assertions at the level of the object language (Sambin et al., 2000, p.~980). According to Sambin et al.\ (2000), ``A logical constant obeys the principle of reflection if it is characterized semantically by an equation binding it to a metalinguistic link between assertions, and if its synthetic inference rules are obtained by solving the equation" (p.279).

We illustrate the principle of reflection using the cases of multiplicative conjunction and additive conjunction. The corresponding cases for disjunction are obtained symmetrically by interchanging the left and right sides of sequents.

  A definitional equation of the tensor connective is as follows: for all $\Delta$, $A\tensor B\vdash \Delta$ if and only if $A, B\vdash \Delta$.

The `if' direction corresponds to the following formation rule:

\begin{center}
\def\fCenter{\ \vdash\ }
\Axiom$A, B\fCenter\Delta$
\RightLabel{$\tensor$-formation}
\UnaryInf$A\tensor B\fCenter\Delta$
\DisplayProof
\end{center}

The `only if' direction corresponds to the following rule:

\begin{center}
\def\fCenter{\ \vdash\ }
\Axiom$A\tensor B\fCenter\Delta$
\RightLabel{implicit $\tensor$-reflection}
\UnaryInf$A, B\fCenter\Delta$
\DisplayProof
\end{center}

This implicit reflection rule contains the connective under consideration in its premise. Therefore, we construct an equivalent rule that does not contain the connective in the premise.
First, we substitute $A \tensor B$ for the context $\Delta$ and obtain the identity. Second, we replace $A$ and $B$ in the sequent $A, B \vdash A \tensor B$ by $\Gamma_1$ and $\Gamma_2$, respectively, using the cut rule.

\begin{center}
\def\fCenter{\ \vdash\ }
\Axiom$\Gamma\fCenter  A$
\Axiom$A, B\fCenter A\tensor B$
\RightLabel{Cut}
\BinaryInf$\Gamma, B\fCenter A\tensor B$
\Axiom$\Gamma' \fCenter B$
\RightLabel{Cut}
\BinaryInf$\Gamma, \Gamma' \fCenter A\tensor B$
\DisplayProof
\end{center}

Finally, we obtain the following rule:

\begin{center}
\def\fCenter{\ \vdash\ }
\Axiom$\Gamma\fCenter  A$
\Axiom$\Gamma'\fCenter B$
\RightLabel{explicit $\tensor$-reflection}
\BinaryInf$\Gamma, \Gamma'\fCenter A\tensor B$
\DisplayProof
\end{center}

In the case of  additive conjunction, we solve the following definitional equation; $\Gamma\vdash A\with B$ if and only if $\Gamma\vdash A$ and $\Gamma\vdash B$ .

The `if' direction corresponds to the following formation rule:

\begin{center}
\def\fCenter{\ \vdash\ }
\Axiom$\Gamma\fCenter  A$
\Axiom$\Gamma\fCenter B$
\RightLabel{$\with$-formation}
\BinaryInf$\Gamma\fCenter A\with B$
\DisplayProof
\end{center}
\medskip

The `only if' direction corresponds to the following rules:

\begin{center}
\def\fCenter{\ \vdash\ }
\Axiom$\Gamma\fCenter A\with B$
\RightLabel{implicit $\with$-reflection 1}
\UnaryInf$\Gamma\fCenter A$
\DisplayProof
\medskip
\def\fCenter{\ \vdash\ }
\Axiom$\Gamma\fCenter A\with B$
\RightLabel{implicit $\tensor$-reflection 2}
\UnaryInf$\Gamma\fCenter B$
\DisplayProof
\end{center}

Substituting $A \with B$ into $\Gamma$ yields the sequents $A \with B \vdash A$ and $A \with B \vdash B$. Replacing $A$ and $B$ by an arbitrary context $\Delta$ using the cut rule, we obtain the following explicit reflection rules:

\begin{center}
\def\fCenter{\ \vdash\ }
\Axiom$A\fCenter\Delta$
\RightLabel{explicit $\with$-reflection 1}
\UnaryInf$A\with B\fCenter\Delta $
\DisplayProof
\medskip
\Axiom$B\fCenter\Delta$
\RightLabel{explicit $\with$-reflection 2}
\UnaryInf$A\with B\fCenter\Delta $
\DisplayProof
\end{center}

\begin{rem}
We do not consider implication in this paper, since its treatment requires nested sequents. Accordingly, implication rules and constants are excluded from the results presented below.
\end{rem}

\medskip

\subsection{Synthetic Connective}
We assume basic background knowledge of linear logic in this paper. The synthetic connectives are explained by the polarity theory of linear logic (Andreoli, 1992). In multiplicative additive linear logic ($\mathsf{MALL}$), connectives and inference rules are classified into two polarity classes: the tensor $\tensor$ and plus $\oplus$ have positive polarity, while with $\with$ and par $\pa$ have negative polarity. The polarity of an inference rule is determined by the polarity of its principal formula. Inference rules for synthetic connectives are obtained by decomposing a sequent that contains only a single formula $A$, where $A$ is composed exclusively of connectives (and meta-variables) of the same polarity.


\begin{figure}[H]
\begin{framed}
\begin{center}
\def\fCenter{\ \vdash\ }
\Axiom$\Gamma_1\fCenter A$
\Axiom$\Gamma_2\fCenter B$
\BinaryInf$\Gamma_1,\Gamma_2\fCenter A\tensor(B\oplus C)$
\DisplayProof
\medskip
\def\fCenter{\ \vdash\ }
\Axiom$\Gamma_1\fCenter A$
\Axiom$\Gamma_2\fCenter C$
\BinaryInf$\Gamma_1,\Gamma_2\fCenter A\tensor(B\oplus C)$
\DisplayProof
\end{center}
\begin{center}
\def\fCenter{\ \vdash\ }
\Axiom$A, B\fCenter \Delta$
\Axiom$A, C\fCenter \Delta$
\BinaryInf$A\tensor(B\oplus C)\fCenter \Delta$
\DisplayProof
\end{center}
\caption{An example of synthetic connective}\label{synthetic example}
\end{framed}
\end{figure}

If synthetic connectives are defined in terms of binary connectives, as in Girard (1999), they are conceptually dependent on those connectives. For this reason, we define inference rules for synthetic connectives directly.

\begin{defn}
A logical connective $\mathcal{C}$ is a synthetic connective if its inference rules are an instance of either  scheme in Figure \ref{synthetic_rule_1} or scheme in Figure \ref{synthetic_rule_2} (where $m\in\mathbb{N}$, $A_{pq}\in\{A_1,\dots,A_n\}$, $1\le p, g\le m$, $1\le  q\le k_g$).

\begin{figure}[H]
\begin{framed}
\begin{center}
\AxiomC{$\Gamma\ \vdash\  A_{11},\dots,A_{1k_{1}}$}
\AxiomC{$\cdots$}
\AxiomC{$\Gamma\ \vdash\ A_{m1},\dots,A_{1k_{m}}$}
\RightLabel{$\mathcal{C}$-right rule}
\TrinaryInfC{$\Gamma\ \vdash\ \mathcal{C}(A_1,\dots, A_n)$}
\DisplayProof
\end{center}
\begin{center}
\def\fCenter{\ \vdash\ }
\Axiom$A_{i1},\dots,A_{ik_{i}}\fCenter \Delta$
\RightLabel{$\mathcal{C}$-left rule}
\UnaryInf$ \mathcal{C}(A_1,\dots, A_m)\fCenter \Delta$
\DisplayProof
\end{center}
\caption{Inference rule scheme I of synthetic connective } \label{synthetic_rule_1}
\end{framed}
\end{figure}

\begin{figure}[H]
\begin{framed}
\begin{center}
\AxiomC{$ A_{11},\dots,A_{1k_{1}}\ \vdash\ \Delta $}
\AxiomC{$\cdots$}
\AxiomC{$A_{m1},\dots,A_{1k_{m}}\ \vdash\ \Delta$}
\RightLabel{$\mathcal{C}$-right rule}
\TrinaryInfC{$\mathcal{C}(A_1,\dots, A_n)\ \vdash\ \Delta$}
\DisplayProof
\end{center}
\begin{center}
\def\fCenter{\ \vdash\ }
\Axiom$\Gamma\fCenter A_{i1},\dots,A_{ik_{i}}$
\RightLabel{$\mathcal{C}$-left rule}
\UnaryInf$\Gamma \fCenter \mathcal{C}(A_1,\dots, A_m)$
\DisplayProof
\end{center}
\caption{Inference rule scheme II of synthetic connective} \label{synthetic_rule_2}
\end{framed}
\end{figure}

\end{defn}

In order to investigate the principle of reflection for synthetic connectives in this paper, we exclude negation from the constituents of synthetic connectives. Consequently, the active formulas appear on the same side of the sequent. Synthetic connectives as presented in Girard (1999) satisfy this property\footnote{
Girard explains that a new connective is not defined by an arbitrary combination of binary connectives; rather, only connectives of the same polarity can form a connective. The reason is that the deducibility of identicals condition fails when connectives of different polarities are combined (Girard, 1999, p.271).}.


\section{Equivalence between two criteria}
Before proceeding, we clarify the terminology used in this section.  A logical connective is identified with a pair of sets of inference rules, whereas a definitional equation is a presentation-dependent object constructed from a chosen set of rules; accordingly, solvability is a property of definitional equations, not of connectives themselves, and is independent of cut-admissibility of the connective.

Hacking's deducibility of identicals condition can be regarded as a procedure for obtaining left (respectively, right) rules from rules on the opposite side. Sambin et al. considered definitional equations by examining all possible combinations of the meta-implication ``yields'' and the meta-conjunction ``and'', and then associated each solvable combination with a logical constant.

By contrast, we proceed in the opposite direction: starting from a given logical connective, we associate it with a corresponding definitional equation and then apply the procedure of the reflection principle to that equation. However, this approach raises the following issue. Suppose that the right rule of the tensor connective is given. The definitional equation obtained from the tensor's right rule is;  $\Gamma_1, \Gamma_2\vdash A\tensor B$ if and only if $\Gamma_1\vdash A$ and $\Gamma_2\vdash B$.  

This equation is unsolvable. By contrast, the definitional equation obtained from the tensor's left rule is solvable. Thus, the choice of which rule we start from cannot be ignored.  In other words, solvability is not determined solely by the connective itself, but can depend on its presentation (i.e. on whether one starts from the left or the right rules), and the visibility condition together with the singleton and non-context-changing requirements is precisely what controls this dependence.

Sambin et al. begin by constructing a definitional equation and then deriving a formation rule together with implicit reflection rules from it. In our approach, we first consider formation rule(s). We regard a definitional equation as expressing, by means of an ``if and only if'', the correspondence between the assumptions and the conclusion of an inference rule.

A set of formation rules need not be a singleton. In such cases, we postulate that the assumptions of multiple rules are combined by ``and'' in the corresponding definitional equation. The class of definitional equations obtained in this way is larger than that considered in Sambin et al. (2000). For example, the left rules of the additive conjunction yield the equation; $A\vdash \Delta$ and $B\vdash \Delta$ if and only if $A\with B\vdash\Delta$,  which is not permitted in Sambin et al.\ (2000). Nevertheless, the set of solvable definitional equations in our framework coincides with that of Sambin et al., since we apply the same solution procedure. Definitional equations that are additional in our terminology are simply unsolvable (in our sense).

For a given connective, two definitional equations can be obtained from its left and right rules. Sambin et al.\ consider only solvable cases, whereas we address the question of when definitional equations are unsolvable. From this analysis, we make explicit the forms of inference rules that render definitional equations solvable. Moreover, our proof of the equivalence between the DoI condition and the reflection principle shows that the DoI condition can be characterized by the success of unification between an assumption of one rule (a formation rule) and the conclusion of another rule (an explicit reflection rule).
In what follows, we assume that the contexts on the side containing the active formulas of a synthetic connective are empty. This is the visibility condition in Basic Logic. According to Sambin et al. (2000), ``a rule satisfies visibility if it operates on a formula (or two formulae) only if it is (they are) the only formula(e), either in the antecedent or in the succedent of a sequent (p.981)".

We denote by $Cxt_{\mathrm{Prem}}(E)$ (respectively, $Cxt_{\mathrm{Concl}}(E)$) the set of contexts occurring in the premises (respectively, conclusions) of the formation rules obtained from a definitional equation $E$.

\begin{rem}
In this paper, a logical connective is taken to be a pair of sets of right rules and left rules such that the main step of cut elimination holds between them. Sambin et al. argue that the cut elimination theorem is required to justify the validity of the cut rule (Sambin et al., 2000, pp.993-994). Accordingly, it is reasonable to restrict our attention to pairs of rules for which the main step of cut elimination is admissible.

\end{rem}

\begin{defn}
Let $\mathcal{C}$ be a logical connective, let  $F$ be a formation rule obtained from $\mathcal{C}$ and let  $E$ be the definitional equation obtained from $F$. $E$ is solvable if the main step of the cut elimination holds between the explicit reflection rule(s) and $F$.

\end{defn}

Satisfiability of the main step of cut elimination and solvability of a definitional equation of a logical connective are not equivalent notions. For example, the definitional equation obtained from the tensor's right rule is unsolvable, although the tensor connective itself satisfies the main step of cut elimination.

\begin{defn}
Let $\mathcal{C}$ be a logical connective and  let $F$ be a formation rule obtained from $\mathcal{C}$.  $F$ is context changing if $Cxt_{Prem}(F)\neq Cxt_{Concl}(F)$ holds. The definitional equation obtained from $F$ is context changing if $F$ is context changing. 

\end{defn}

\begin{lem}\label{context changing}
Let $\mathcal{C}$ be a logical connective and  let $E$ be its one of definitional equation. If $E$ is context changing, then $E$ is unsolvable.
\end{lem}

\begin{proof}
We assume that $E$ is context changing. Let $F$ be one of formation rules of $E$. $F$ is written as follows (where $A_{ij}\in\{A_1,\dots,A_n\}$, $1\le i, l\le m$, $1\le j, i'\le k_l$, $\bigcup \Gamma'\subsetneq\bigcup\Gamma$).

\begin{center}
\AxiomC{$\Gamma_1\ \vdash\  A_{11},\dots,A_{1k_{1}}$}
\AxiomC{$\cdots$}
\AxiomC{$\Gamma_m\ \vdash\ A_{m1},\dots,A_{1k_{m}}$}
\TrinaryInfC{$\Gamma'_1,\dots,\Gamma'_k\ \vdash\ \mathcal{C}(A_1,\dots, A_n)$}
\DisplayProof
\end{center}

The implicit reflection rule has the following form.
\begin{center}
\def\fCenter{\ \vdash\ }
\Axiom$\Gamma'_1,\dots,\Gamma'_k\fCenter \mathcal{C}(A_1,\dots, A_n)$
\UnaryInf$\Gamma_i\fCenter  A_{ii'}$
\DisplayProof
\end{center}

The explicit reflection rule has the following form.
\begin{center}
\def\fCenter{\ \vdash\ }
\Axiom$A_{ii'}\fCenter \Delta$
\UnaryInf$ \mathcal{C}(A_1,\dots, A_n)\fCenter \Delta$
\DisplayProof
\end{center}

Then,  we obtain $\Gamma'_1,\dots,\Gamma'_k\vdash\Delta$ by the cut on $\mathcal{C}(A_1,\dots, A_n)$ and $\Gamma_i\vdash \Delta$ by the cut on $A_{ii'}$. 
Therefore, the main step of the cut elimination does not hold. Hence, $E$ is unsolvable.

\end{proof}

The following lemma was pointed out by  (Schroeder-Heister, 2012, p.498).

\begin{lem}\label{multi formation}

Let $\mathcal{C}$ be a logical connective and $E$ be a definitional equation of $\mathcal{C}$. If $E$ has several formation rules, then $E$ is unsolvable.
\end{lem}
\begin{proof} 
If $E$ is context changing, then it is unsolvable by Lemma \ref{context changing}. Hence, we assume that  contexts of  formation rules are the same in its premise and conclusion. Let formation rules of $E$ be $F_1,\dots,F_m$  (where $m$ is a natural number).
We assume that $F_i$ ($1\le i\le m$) has the following form  (where $m\in\mathbb{N}$, $A^{i}_{pq}\in\{A_1,\dots,A_n\}$, $1\le p, g\le m$, $1\le  q\le k_g$).

\begin{center}
\AxiomC{$\Gamma\ \vdash\  A^{i}_{11},\dots,A^{i}_{1k_{1}}$}
\AxiomC{$\cdots$}
\AxiomC{$\Gamma\ \vdash\ A^{i}_{m1},\dots,A^{i}_{1k_{m}}$}
\TrinaryInfC{$\Gamma\ \vdash\ \mathcal{C}(A_1,\dots, A_n)$}
\DisplayProof
\end{center}

We assume that an another formation rule $F_j$ ($j\neq i$ and $1\le j\le m$) has the following form (where $A^{j}_{pq}\in\{A_1,\dots,A_n\}$, $1\le p, g\le m$, $1\le q \le l_g$).

\begin{center}
\AxiomC{$\Gamma\ \vdash\  A^{j}_{11},\dots,A^{j}_{1l_{1}}$}
\AxiomC{$\cdots$}
\AxiomC{$\Gamma\ \vdash\ A^{j}_{m1},\dots,A^{j}_{1l_{m}}$}
\TrinaryInfC{$\Gamma\ \vdash\ \mathcal{C}(A_1,\dots, A_n)$}
\DisplayProof
\end{center}

An assumption of the explicit reflection rule obtained from $F_i$ is $A^{i}_{ss'}\vdash\Delta$ ($1\le s, g\le m$, $1\le s'\le k_g$), and an assumption of the explicit reflection rule obtained from $F_j$ is $A^{j}_{tt'}\vdash\Delta$ ($1\le t, g\le m$, $1\le t'\le k_g$). There is a pair of different formulas in $A^{i}_{ss'}$ and $A^{j}_{tt'}$. Otherwise, $F_i=F_j$ holds and it contradicts our assumption. The main step of the cut elimination between a pair of different formulas in $A^{i}_{ss'}$ and $A^{j}_{tt'}$ does not hold. Hence, $E$ is unsolvable.


\end{proof}


\begin{thm}
A logical connective $\mathcal{C}$ satisfies the deducibility of identicals condition if and only if $\mathcal{C}$ satisfies the reflection principle.
\end{thm}
\begin{proof}

(if-direction) We assume that $\mathcal{C}$ satisfies the reflection principle. We apply explicit reflection rule(s) to axiom(s). Let $S_0$ be the resulting sequent(s). One side of the assumption of explicit reflection rule(s) contains one propositional variable and the other side contains context variable(s). We trivialize the context variable and obtain an axiom. Hence, the application of the explicit reflection rule(s) is possible.  Next, we consider a formation rule.  A formation rule contains the logical connective in one side and a context in the other side. We replace this context with the formula $\mathcal{C}(A_1,\dots, A_n)$. Let  $S_1$ be the sequent(s) in assumptions of this instance.  Matching between one side of sequent in the assumption of a formation rule and one in conclusion of an explicit reflection rule succeeds because of the construction of explicit reflection rules. Hence, the matching between $S_0$ and $S_1$ succeeds and the deducibility of identicals condition holds.  

(only-if direction) By definition of DoI, the sequent $\mathcal{C}(A_1,\dots,A_n)\vdash\mathcal{C}(A_1,\dots \\
,A_n)$ is obtainable by applying (several) left (resp. right) rule(s) and then a right (resp. left) rule. We regard the rule(s) we first apply as the explicit reflection rule and the rules we secondly apply as the formation rule.  
We replace  contexts in a definitional equation with the formula  $\mathcal{C}(A_1,\dots,A_n)$. From one direction of a definitional equation, we obtain a sequent $S$ because the identity  $\mathcal{C}(A_1,\dots,A_n)\vdash\mathcal{C}(A_1,\dots,A_n)$ holds. We can replace active subformulas of  $\mathcal{C}(A_1,\dots,A_n)$ in $S$ with contexts by the cut. Hence, the reflection principle holds.
\end{proof}

\begin{cor}\label{DoI_equivalence}
A pair of sets of inference rules $(L, R)$ satisfies the main step of cut elimination and the deducibility of identicals condition if and only if it satisfies the reflection principle.
\end{cor}

The uniqueness condition is easily followed from the definitional equations.

\begin{lem}\label{get_uniqueness}
If the definitional equation of $\mathcal{C}(A_1,\dots, A_n)$ is solvable, then $\mathcal{C}(A_1,\dots, A_n)$ satisfies uniqueness condition.
\end{lem}
\begin{proof}
Let $\mathcal{C}^{\ast}(A_1,\dots, A_n)$ be a connective that has the same inference rule as $\mathcal{C}$. We consider the definitonal equation of $\mathcal{C}^{\ast}$. The right side of it is the same as $\mathcal{C}$.
It is possible to compose these definitional equations because these are solvable. We substitute the context variable in the obtained equation for  $\mathcal{C}(A_1,\dots, A_n)$ and $\mathcal{C}^{\ast}(A_1,\dots, A_n)$, in turn.
The one side in the equation is identity and we obtain  $\mathcal{C}(A_1,\dots, A_n)\vdash\mathcal{C}^{\ast}(A_1,\dots, A_n)$ and $\mathcal{C}^{\ast}(A_1,\dots, A_n)\vdash \mathcal{C}(A_1,\dots, A_n)$.

\end{proof}

\begin{cor}
A pair of sets of inference rules $(L, R)$ satisfies the main step of cut elimination and uniqueness condition if and only if it satisfies the reflection principle.
\end{cor}

\begin{proof}
It follows from Corollary \ref{DoI_equivalence} and Lemma \ref{get_uniqueness}.
\end{proof}

We give a characterization of synthetic connectives.

\begin{thm}\label{charact of synn}
A logical connective $\mathcal{C}$ (its arity is an arbitrary natural number $n$) satisfies the reflection principle if and only if  $\mathcal{C}$ is a synthetic connective.
\end{thm}
\begin{proof}
We assume that  $\mathcal{C}$ appears in the succedent of a sequent.  We also assume that $\mathcal{C}$ satisfies the reflection principle. By Lemma \ref{multi formation}, the set of the formation rules of $\mathcal{C}$ is singleton. By Lemma \ref{context changing}, the contexts in the assumption and conclusion are the same. Hence,  one of the rules of $\mathcal{C}$ has the following form (where $m\in\mathbb{N}$, $A_{ij}\in\{A_1,\dots,A_n\}$, $1\le i, l\le m$, $1\le j, i'\le k_l$).

\begin{center}
\AxiomC{$\Gamma\ \vdash\  A_{11},\dots,A_{1k_{1}}$}
\AxiomC{$\cdots$}
\AxiomC{$\Gamma\ \vdash\ A_{m1},\dots,A_{1k_{m}}$}
\TrinaryInfC{$\Gamma\ \vdash\ \mathcal{C}(A_1,\dots, A_n)$}
\DisplayProof
\end{center}

Hence, $\mathcal{C}$ is a synthetic connective.
The reverse direction follows by easy calculation.

\end{proof}
Sambin et al. conjectured that only six definitional equations - corresponding to two conjunctions, two disjunctions, and two implications - are solvable (Sambin et al., 2000, p.985). According to our result, in addition to these six equations, definitional equations obtained from synthetic connectives are also solvable. Of course, if synthetic connectives are regarded as reducible to conjunctions and disjunctions, then Sambin et al.'s conjecture can be maintained.

\section{Conclusion}
We have proved the equivalence between the principle of reflection and the deducibility of identicals condition in the absence of implication (which requires nested sequents).  Moreover, we have shown that a logical connective satisfies the principle of reflection if and only if it is a synthetic connective. Finally, we have demonstrated that the property of having a singleton set of formation rules is essential for satisfying the principle of reflection.






\begin{thebibliography}{20}
\bibitem{Andreoli}J.-M. Andreoli, \textit{Logic programming with focusing proofs in linear logic}, Journal of Logic and Computation, vol. 2, no. 3, 1992,  pp. 297-347.
\bibitem{Belnap}N.D.Belnap, \textit{Tonk, Plonk and Plink}, Analysis, vol. 22, No. 6, 1962, pp. 130-134.
\bibitem{Girard}J.-Y. Girard, \textit{Linear Logic}, Theoretical Computer Science, vol. 50, 1987, pp.1-102.
\bibitem{Girard 1999}J.-Y. Girard, \textit{On the meaning of logical connectives I: syntax vs. semantics}, in Berger, Schwichtenberg (Eds), Computational logic, Springer, 1999, pp. 215-272.
\bibitem{Girard 2001}J.-Y. Girard, \textit{Locus solum: From the rules of logic to the logic of rules}, Mathematical Structures in Computer Science, vol. 11, Issue 3, 2001, pp.301-506.
\bibitem{Hacking}I. Hacking, \textit{What is Logic?}, Journal of Philosophy, vol.76 Issue 6, 1979, pp. 285-319.
\bibitem{Naibo Petrolo}A. Naibo and M. Petrolo, \textit{Are Uniqueness and Deducibility of Identicals the Same?}, Theoria, vol.81, 2015, pp.143-181.
\bibitem{Sambin}G.Sambin, G.Battilotti and C. Faggian, \textit{Basic logic: reflection, symmetry, visibility}, The Journal of Symbolic Logic, vol. 65, no.3, 2000, p.979-1013.
\bibitem{Schroeder BL}P. Schroeder-Heister, \textit{Definitional Reflection and Basic Logic}, Annals of Pure and Applied Logic, vol. 164, Issue 4,  2013, pp.491-501
\end{thebibliography}
\end{document}